\documentclass[12pt]{article}
\usepackage{geometry}                
\usepackage{graphicx}
\usepackage{amssymb}
\usepackage{epstopdf}
\def\eqalign#1{\null\,\vcenter{\openup\jot \mathsurround=0pt \ialign{\strut
     \hfil$\displaystyle{##}$&$ \displaystyle{{}##}$\hfil \crcr#1\crcr}}\,}
\def\K{{\cal K}}
\def\H{{\cal H}}
\usepackage{amsthm}

\newtheorem{example}{Example}
\newtheorem{proposition}{Proposition}

\title{Discrete gradient methods have an energy conservation law}
\author{Robert I McLachlan and G R W Quispel\\[4mm]
{\em Dedicated to Arieh Iserles on the occasion of his 65th birthday}}
\def\Da{\overline{H_z} }
\def\Db{\overline{H_{z_x}} }
\def\dg{\overline{\nabla}}
\def\R{{\mathbb R}}

\begin{document}
\maketitle
\begin{abstract}
We show for a variety of classes of conservative PDEs that discrete gradient methods designed to have a conserved quantity (here called energy) also have a time-discrete conservation law. The discrete conservation law has the same conserved density as the continuous conservation law, while its flux is found by replacing all derivatives of the conserved density appearing in the continuous flux by discrete gradients. 
\end{abstract}
The discrete gradient method \cite{gonzalez,mc-qu-ro,qu-tu,SimoTarnow92} is a powerful and general technique for constructing integral-preserving integrators. It includes projection methods and the average vector field method \cite{qu-mc} as special cases. In particular, it is useful for constructing energy-preserving integrators. 

For PDEs, a conservation law is more fundamental than a conserved quantity, for the former implies the latter only in the presence of suitable boundary conditions, and the flux is (usually) a local quantity. Discrete gradient methods have been widely used to construct energy-preserving integrators for PDEs and their semidiscretizations; see \cite{avfpde} for a survey. In this paper we show that discrete gradient methods also preserve energy conservation laws (ECLs) when the symplectic structure is constant. Discrete gradient methods are not usually symplectic or variational and the resulting conservation laws developed in this paper are not instances of Noether's theorem. Despite their nonvariational nature, the use of `conservative' discretizations of systems of conservation laws of mass, momentum and so on---that is, of discretizations with discrete conservation laws---is widespread in numerical methods for PDEs. The typical case is that the conserved density is (one of) the dependent variable(s): for example,
$$ u_t + f(u)_x = 0$$
is a conservation law with conserved density $u$ and flux $f(u)$. Euler's method
$$\frac{u_1 - u_0}{\Delta t} + f(u_0)_x = 0$$
is a time discretization in the form of a (time-)discrete conservation law. (The full discretization
$$\frac{u_1 - u_0}{\Delta t} + \Delta( f(u_0) )$$
where $\Delta$ is any difference operator in $x$, is a fully discrete conservation law, but the spatial discretization is covered by discrete versions of Noether's theorem (see, e.g., \cite{hy-ma} and references therein) and is not the subject of this paper.) Another known case is when a conservation law arises in a system of Euler--Lagrange equations in which the Lagrangian admits a symmetry acting only on the dependent variables. If the symmetry is fairly simple, e.g., if it is linear, then there may exist symplectic or discrete Lagrangian methods that preserve the symmetry and which will have a discrete conservation law. See \cite{ry-mc-fr} where this situation is illustrated for the phase symmetry of the nonlinear Schr\"odinger equation.

Discrete gradient methods are a way of ensuring that a time discretization has a conservation law when the conserved density, that we call energy in this paper, is not or cannot be one of the dependent variables, in a way that is independent of any variational structure or symmetry.

A discrete gradient on a vector space $V$ with inner product $a^\top b$ is a function satisfying the discrete gradient axiom
$$ (z_1-z_0)^\top \dg H(z_0,z_1) = H(z_1) -H(z_0)$$
for all smooth functions $H\colon V\to{\mathbb R}$ and all $z_0$, $z_1\in V$. For consistency in the continuous limit, generally one also requires
$$ \dg H(z_0,z_0) = \nabla H(z_0),$$
although for conservation properties only the discrete gradient axiom is needed. A common discrete gradient is the Average Value discrete gradient given by
$$ \dg H(z_0,z_1) = \int_0^1 \nabla H(\xi z_1 + (1-\xi) z_0)\, {\rm d}\xi.$$

In coordinates we write $\overline{H_y}$ for the $y$-component of $\overline\nabla H$. If $z=(x,y)$, then the discrete gradient axiom becomes
$$ (x_1 - x_0)^\top\overline{H_x} + (y_1-y_0)^\top \overline{H_y} = H(x_1,y_1) -H(x_0,y_0).$$

A discrete gradient method for the system $z_t = K(z) \nabla H(z)$, $K^\top = -K$, is a map $z_0\mapsto z_1$ given by
$$\frac{z_1-z_0}{\Delta t} = \overline K(z_0,z_1)\overline\nabla H(z_0,z_1)$$
where $\overline K^\top = -\overline K$ and (for consistency) $\overline K(z_0,z_0) = K(z_0,z_0)$.

We first give an explicit calculation of the ECL for first-order canonical Hamiltonian PDEs in one space dimension and their discrete gradient time discretizations.

\begin{proposition}
Let $z(x,t)\in\R^n$, let $K$ be an $n\times n$ antisymmetric matrix and let $\cal H$ be a Hamiltonian of the form
${\cal H} = \int\! H(z,z_x)\, {\rm d}x$, and consider the equations of motion
$$z_t =K\frac{\delta {\cal H}}{\delta z} = K(H_z - \partial_x H_{z_x}).$$
The discrete gradient method
$$ \frac{z_1 - z_0}{\Delta t} = K\left(\Da - \partial_x \Db\right)$$
has a discrete ECL. Its conserved energy density, $H$, is the same as that of the PDE. Its flux is  the same bilinear form as that of the continuous ECL, but with the gradients of the energy density replaced by their discrete gradients. 
\end{proposition}
\begin{proof}
The continuous ECL can be found as follows.
$$\eqalign{ H_t&= H_z^\top z_t + H_{z_x}^\top z_{xt} \cr
&= H_z^\top K(H_z - \partial_x H_{z_x}) + H_{z_x}^\top K(H_z - \partial_x H_{z_x})_x \cr
&= \left(H_{z_x}^\top K H_z - H_{z_x}^\top K\partial_x H_{z_x}\right)_x \cr
}$$
giving the energy conservation law $H_t + F_x = 0$ where
$$ F = -H_{z_x}^\top K H_z + H_{z_x}^\top K\partial_x H_{z_x}.$$
Applying a discrete gradient method in $(z,z_x)$ to the PDE gives the time discretization
$$ \frac{z_1 - z_0}{\Delta t} = K\left(\Da - \partial_x \Db\right).$$
The change in energy density over one time step is 
$$\eqalign{
 \frac{H(z_1,z_{1x}) - H(z_0,z_{0x})}{\Delta t}&= 
 \Da^\top \frac{z_1-z_0}{\Delta t} + \Db^\top\frac{z_{1x}-z_{0x}}{\Delta t} \cr
 &= \Da^\top  K (\Da - \partial_x \Db) + \Db^\top K(\Da-\partial_x \Db)_x \cr
 &= \left(\Db^\top K\Da - \Db^\top K\partial_x\Db\right)_x \cr}$$
 which establishes the result.
 \end{proof}
 Note that the discrete calculation exactly follows the continuous one, the key step being the discrete gradient axiom. The calculation is entirely local and does not require integration by parts.
 
\begin{example}\rm
The nonlinear wave equation $q_t = p$, $p_t = q_{xx}-V'(q)$ fits the framework of Proposition 1 with
$n=2$, $K=\left(\matrix{0 & 1 \cr -1 & 0}\right)$, and $z = (q,p)^\top$. The energy conservation law is
$$ H(q,p)_t + F(q,p)_x = 0$$
where
$$H = \frac{1}{2}p^2 + \frac{1}{2}q_x^2 + V(q),\quad F = - pq_x,$$
and the discrete gradient ECL is
$$ \frac{H(q_1,p_1) - H(q_0,p_0)}{\Delta t} + F(\overline q,\overline p)_x =0$$
where $\overline q = (q_0+q_1)/2$ and $\overline p = (p_0+p_1)/2$.
\end{example}

\begin{proposition}
Let $\K$ be any constant (i.e., independent of $z$) skew-adjoint differential operator and 
let $\H = H(z,z_x,z_{xx},\dots)$ be any differential function of $z$ (and possibly $x$, although
we suppress the $x$). Let $\K$ and $\H$ determine the Poisson system
$$ z_t = \K\frac{\delta {\cal H}}{\delta z} = \K E(H)$$
where
$$ E(H) := H_z - \partial_x H_{z_x} + \partial_{xx} H_{z_{xx}} - \dots$$
is the Euler operator applied to $H$.
The flux associated to the conserved density $H$ is a quadratic form in the derivatives of $H$. 
A discrete gradient method applied to this system has a discrete ECL with conserved density $H$ and
flux given by the same quadratic form with derivatives of $H$ replaced by discrete gradients.
\end{proposition}
\begin{proof}
From skew-adjointness of $\K$ we can write
$$ a^\top {\cal K}b + b^\top {\cal K}a = \partial_x S(a,b).$$
(For example, $S(a,b)=a^\top \K_2 b$ for ${\cal K} = \K_1 + \K_2\partial_x$.)
Then the flux for the energy conservation law is the quadratic form \cite{vanneste}
$$ S(E(H),E(H)) - A({\cal K}E(H),H)$$
where 
$$ A(Q,L) = \sum_{n=1}^\infty\sum_{i=0}^{n-1}(-1)^i \partial_x^{n-1-i}Q\partial_x^i\frac{\partial L}{\partial z_x^n}.$$
The calculation of the discrete ECL now follows as in Proposition 1.
\end{proof}

In the literature there has been much attention paid to multisymplectic discretizations of variational PDEs \cite{br-re}. These often take the multi-Hamiltonian form
$$ K z_t + L z_x = \nabla S(z)$$
which has a multisymplectic conservation law $({\rm d}z\wedge K {\rm dz})_t + ({\rm d}z\wedge L {\rm dz})_x = 0$. Multisymplectic methods preserve a discretization of this (differential) conservation law, but do not have a energy conservation law. (This is the PDE analogue of the fact that a integrator for a general Hamiltonian ODE cannot be both symplectic and energy-preserving). Our next result covers this case. As in Proposition 2, it could be extended to differential operators $K$ and arbitrary differential polynomials $H$, but the case considered here shows the basic structure.

\begin{proposition}
Let $\H = H(z,z_x)$ and let $K$ be a (possibly singular) antisymmetric matrix.
Then PDE
$$ K z_t = \frac{\delta \H}{\delta z}$$
has an ECL with conserved density $H$ and flux given by a bilinear form in $z_t$ and the derivatives of $H$.
The discrete gradient method
$$ K\frac{z_1-z_0}{\Delta t} = \overline H_z - \partial_x \overline H_{z_x}$$
has a discrete ECL with conserved density $H$ and flux given by the same bilinear form with $z_t$ replaced by $(z_1-z_0)/\Delta t$ and the derivatives of $H$ replaced by discrete gradients.
\end{proposition}
\begin{proof}
First we note that multi-Hamiltonian PDEs take the given form with $H = S(z)- \frac{1}{2}z^\top L z_x$ and $ \frac{\delta \H}{\delta z} = \nabla S(z)-Lz_x$. The continuous ECL is found as follows:
$$ \eqalign{
0 &= z_t^\top K z_t\cr
& = z_t^\top(H_z - \partial_x H_{z_x}) \cr
&= (H_z^\top z_t + H_{z_x}^\top z_{xt}) - (H_{z_x}^\top z_{xt} + z_{t}^\top \partial_x H_{z_x})\cr
&= H_t - (H_{z_x}^\top z_t)_x.}
$$
The key step is the use of the product rule for the derivative of the flux $H_{z_x}^\top z_t$. Note that there is no requirement for $K$ to be invertible.
If $K$ is invertible then one can solve and substitute for $z_t$ to give the same flux as in Proposition 1.
For the discrete gradient method
we have
$$
\eqalign{
0 &= (z_1 - z_0)^\top K (z_1 - z_0)/\Delta t \cr
&= (z_1-z_0)^\top (\overline{H_z} - \partial_x \overline{H_{z_x}}) \cr
&= (z_1-z_0)^\top\overline{ H_z}+ (z_{1x}-z_{0x})^\top \overline{ H_{z_x}} - ((z_{1x}-z_{0x})^\top \overline{ H_{z_x} }
+ (z_1 - z_0)^\top \partial_x \overline{ H_{z_x}}) \cr
&= H(z_1,z_{1x}) - H(z_0,z_{0x}) - (\overline{H_{z_x}}^\top(z_{1}-z_{0}))_x \cr
&= \Delta_t H - (\overline{H_{z_x}}^\top\Delta_t z)_x,\cr
}$$
establishing the result.
\end{proof}

We now consider the fully discrete case. The existence and derivation of conservation laws for variational and Hamiltonian discretizations is the subject of discrete versions of Noether's theorem, see e.g. \cite{hy-ma}. We restrict ourselves here to the observation that if the semidiscretization has a semidiscrete ECL, then a discrete gradient method applied to this semidiscretization will have a fully discrete ECL. In the following proposition, the extension to $\H=\int H(z,z_x)$ is routine, we omit the $z$-dependence for clarity.

\begin{proposition}
Let $z(x,t)\in\R^n$, let $K$ be an $n\times n$ antisymmetric matrix and let $\cal H$ be a Hamiltonian of the form
${\cal H} = \int H(z_x)\, {\rm d}x$ with equations of motion
$$z_t =K\frac{\delta {\cal H}}{\delta z} = - K\partial_x H_{z_x}.$$
Consider a finite difference discretization with $\H=\int H(z_x)\, {\rm d}x$ discretized to $\H_d = \sum_i H(\Delta z_i)$ where $\Delta z_i := z_{i+1}-z_i$. The semidiscretization 
$$ z_t = K \nabla\H_d$$
has a semidiscrete ECL and the discrete gradient methods applied to this semidiscretization has a fully discrete ECL.
\end{proposition}
\begin{proof}
The semidiscretization is
$$\eqalign{
(z_i)_t &= K(\nabla\H_d)_i \cr
&= K(\nabla H(\Delta z_{i-1})- \nabla H(\Delta z_i)) \cr
&= -K \Delta\nabla H(\Delta z_{i-1}) \cr
}$$
and the semidiscrete ECL is found as
$$\eqalign{
 (H(\Delta z_i))_t &= \nabla H(\Delta z_i)^\top (\Delta z_{i})_t \cr
&= -\nabla H(\Delta z_i)^\top\Delta  K \Delta  \nabla H(\Delta z_{i-1}) \cr
&= -\Delta \nabla H(\Delta z_i)^\top K \Delta \nabla H(\Delta z_i) - \nabla H(\Delta z_i)^\top K \Delta \nabla H(\Delta z_{i-1})\cr
&= -\Delta(\nabla H(\Delta z_i) K \Delta \nabla H(\Delta z_{i-1})) \cr
}$$
where the first term in the third line is zero and the last line follows because of the discrete product rule
$$ \Delta(a_i b_{i-1}) = b_i \Delta a_i + a_i \Delta b_{i-1}.$$
The parallel with the continuous ECL
$$ H_t = -H_{z_x}^\top K \partial _x H_{z_x}$$ 
is clear.
The discrete gradient method
$$\frac{z_1 - z_0}{\Delta t} = -K \Delta \overline{H_{z_x}}(\Delta z_{i-1})$$
has fully discrete ECL
$$\frac{H(z_1)-H(z_0)}{\Delta t} = 
- K \Delta(\overline{H_{z_x}}(\Delta z_i)^\top K \Delta \overline{H_{z_x}}(\Delta z_{i-1})),$$
the calculation proceeding exactly as in the spatially continuous case.
\end{proof}

\section*{Acknowledgements}
This research was supported by the Marsden Fund of the Royal Society of New Zealand and the Australian Research Council.

\end{document}